\documentclass[11pt,reqno]{amsart}

\usepackage{graphicx}

\usepackage{xspace}
\usepackage{amsmath,amsthm,amsfonts,amssymb,latexsym,mathrsfs,color,extarrows}
\usepackage{hyperref}

\newtheorem{theorem}{Theorem}
\newtheorem{corollary}[theorem]{Corollary}

\newtheorem{conjecture}[theorem]{Conjecture}

\newtheorem{lemma}[theorem]{Lemma}
\newtheorem{definition}[theorem]{Definition}
\newtheorem{example}[theorem]{Example}
\newtheorem{problem}[theorem]{Problem}

\newcommand{\CLS}{{\rm CLS\,}}

\newcommand{\MM}{{\rm M\,}}
\newcommand{\Jc}{{\rm Jc\,}}
\newcommand{\Lc}{{\rm Lc\,}}

\newcommand{\LS}{{\rm LS\,}}

\newcommand{\JS}{{\rm JS\,}}

\newcommand{\cyc}{{\rm cyc\,}}

\newcommand{\msn}{\mathfrak{S}_n}
\newcommand{\CLSS}{\mathcal {CLS}}
\newcommand{\LSS}{\mathcal{LS}}

\DeclareMathOperator{\N}{\mathbb{N}}

\newcommand{\stirling}[2]{\genfrac{[}{]}{0pt}{}{#1}{#2}}
\newcommand{\Stirling}[2]{\genfrac{\{}{\}}{0pt}{}{#1}{#2}}

\title[Legendre-Stirling numbers]{On certain combinatorial expansions of the Legendre-Stirling numbers}
\author[S.-M.~Ma]{Shi-Mei Ma}
\address{School of Mathematics and Statistics,
        Northeastern University at Qinhuangdao,
         Hebei 066000, P.R. China}
\email{shimeimapapers@163.com (S.-M. Ma)}
\author[J.~Ma]{Jun Ma}
\address{Department of mathematics, Shanghai jiao tong university, Shanghai, P.R. China}
\email{majun904@sjtu.edu.cn(J.~Ma)}
\author[Y.-N. Yeh]{Yeong-Nan Yeh}
\address{Institute of Mathematics,
        Academia Sinica, Taipei, Taiwan}
\email{mayeh@math.sinica.edu.tw (Y.-N. Yeh)}
\begin{document}

\maketitle
\begin{abstract}
The Legendre-Stirling numbers of the second kind were introduced by Everitt et al. in the spectral
theory of powers of the Legendre differential expressions.
In this paper, we provide a combinatorial code for Legendre-Stirling set partitions. As an application, we
obtain combinatorial expansions of the Legendre-Stirling numbers of both kinds.
Moreover, we present grammatical descriptions of the Jacobi-Stirling numbers of both kinds.
\bigskip

\end{abstract}
\date{\today}
\section{Introduction}

Let $\ell[y](t)=-(1-t^2)y''(t)+2ty'(t)$ be the Legendre differential operator.
Then the Legendre polynomial $y(t)=P_n(t)$ is an eigenvector for the differential operator $\ell$ corresponding to $n(n+1)$, i.e.,
$\ell[y](t)=n(n+1)y(t)$.
Following Everitt et al.~\cite{Everitt02}, for $n\in \N$, the {\it Legendre-Stirling numbers $\LS(n,k)$ of the second kind}
appeared originally as the coefficients in the expansion of the $n$-th composite power of $\ell$, i.e., 
\begin{align*}
\ell^n[y](t)&=\sum_{k=0}^n(-1)^k\LS(n,k)((1-t^2)^ky^{(k)}(t))^{(k)}.
\end{align*}
For each $k\in \N$, Everitt et al.~\cite[Theorem 4.1)]{Everitt02} obtained that
\begin{equation}\label{vertical-GF}
\prod_{r=1}^k\frac{1}{1-r(r+1)x}=\sum_{n=0}^\infty\LS(n,k)x^{n-k},~\left(|x| \leq \frac{1}{k(k+1)} \right),
\end{equation}
$$\LS(n,k)=\sum_{r=0}^k(-1)^{r+k}\frac{(2r+1)(r^2+r)^n}{(r+k+1)!(k-r)!}.$$
According to~\cite[Theorem 5.4]{Andrews11}, the numbers $\LS(n,k)$
have the following horizontal generating function
\begin{equation}\label{xnJsnkz01}
x^n=\sum_{k=0}^n\LS(n,k)\prod_{i=0}^{k-1}(x-i(1+i)).
\end{equation}
It follows from~\eqref{xnJsnkz01} that
the numbers $\LS(n,k)$ satisfy the recurrence relation
\begin{equation*}
\LS(n,k)=\LS(n-1,k-1)+k(k+1)\LS(n-1,k).
\end{equation*}
with the initial conditions $\LS(n,0)=\delta_{n,0}$ and $\LS(0,k)=\delta_{0,k}$, where $\delta_{i,j}$ is the Kronecker's symbol.

By using~\eqref{vertical-GF}, Andrews et al.~\cite[Theorem~5.2]{Andrews11} derived that the numbers $\LS(n,k)$
satisfy the vertical recurrence relation $$\LS(n,j)=\sum_{k=j}^n\LS(k-1,j-1)(j(j+1))^{n-k}.$$
A particular values of $\LS(n,k)$ is provided at the end of~\cite{Andrews09}:
\begin{equation}\label{LSnk01}
\LS(n+1,n)=2\binom{n+2}{3}.
\end{equation}
In~\cite[Eq.~(19)]{Egge10}, Egge found that
\begin{align*}
\LS(n+2,n)&=40\binom{n+2}{6}+72\binom{n+2}{5}+36\binom{n+2}{4}+4\binom{n+2}{3}.
\end{align*}
Using the triangular recurrence relation $\binom{n+1}{k}=\binom{n}{k}+\binom{n}{k-1}$, we get
\begin{equation}\label{LSnk02}
\LS(n+2,n)=40\binom{n+3}{6}+32\binom{n+3}{5}+4\binom{n+3}{4}.
\end{equation}
Egge~\cite[Theorem 3.1]{Egge10} showed that for $k\geq 0$, the quantity $\LS(n+k,n)$ is a polynomial of degree $3k$ in $n$ with leading coefficient $\frac{1}{3^kk!}$.

This paper is a continuation of~\cite{Egge10}, and it is motivated by the following problem.
\begin{problem}\label{problem01}
Let $k$ be a given nonnegative integer. Could the numbers $\LS(n+k,n)$ be expanded
in the binomial basis?
\end{problem}

The paper is organized as follows.
In Section~\ref{Section02}, by introducing a combinatorial code for Legendre-Stirling set partitions, we give a solution of Problem~\ref{problem01}.
Moreover, we get a combinatorial expansion of the Legendre-Stirling numbers of the first kind.
In Section~\ref{Section03},
we present grammatical interpretations of Jacobi-Stirling numbers of both kinds.
\section{Legendre-Stirling set partitions}\label{Section02}
\hspace*{\parindent}

The combinatorial interpretation of the Legendre-Stirling numbers $\LS(n,k)$ of the second kind was first given by
Andrews and Littlejohn~\cite{Andrews09}. For $n\geq 1$, let $\MM_n$ denote the multiset $\{1,\overline{1},2,\overline{2},\ldots,n,\overline{n}\}$,
in which we have one unbarred copy and one barred copy of each integer $i$, where $1\leq i\leq n$.
Throughout this paper, we always assume that the elements of $\MM_n$ are ordered by
$$\overline{1}=1<\overline{2}=2<\cdots <\overline{n}=n.$$
A {\it Legendre-Stirling set partition} of $\MM_n$ is a set partition of $\MM_n$ with $k+1$
blocks $B_0,B_1,\ldots, B_k$ and with the following rules:
\begin{itemize}
  \item [\rm ($r_1$)] The `zero box' $B_0$ is the only box that may be empty and it may not contain
both copies of any number;
  \item [\rm ($r_2$)] The `nonzero boxes' $B_1,B_2,\ldots,B_k$ are indistinguishable and each is non-empty. For any $i\in [k]$, the box $B_i$
contains both copies of its smallest element and does not contain both copies of any other number.
\end{itemize}
Let $\LSS(n,k)$ denote the set of Legendre-Stirling set partitions of $\MM_n$ with one zero box and $k$ nonzero boxes.
The {\it standard form} of an element of $\LSS(n,k)$ is written as
$$\sigma=B_1B_2\cdots B_kB_0,$$ where $B_0$ is the zero box and the minima of $B_i$ is less than that of $B_j$ when $1\leq i<j\leq k$. Clearly, the minima of $B_1$ are $1$ and $\overline{1}$.
Throughout this paper we always write $\sigma\in \LSS(n,k)$ in the
standard form. As usual,
we let angle bracket symbol $<i,j,\ldots>$ and curly bracket symbol $\{k,\overline{k},\ldots\}$ denote the zero box and nonzero box, respectively.
In particular, let $<>$ denote the empty zero box.
For example, $\{1,\overline{1},3\}\{2,\overline{2}\}<\overline{3}>\in\LSS(3,2)$.
A classical result of Andrews and Littlejohn~\cite[Theorem 2]{Andrews09} says that $$\LS(n,k)=\#\LSS(n,k).$$

We now provide a combinatorial code for Legendre-Stirling partitions ($\CLS$-sequence for short).
\begin{definition}
We call $Y_n=(y_1,y_2,\ldots,y_n)$ a {\it $\CLS$-sequence} of length $n$ if $y_1=X$ and $$y_{k+1}\in \{X,A_{i,j},B_s,\overline{B}_s,1\leq i,j,s\leq n_x(Y_k),i\neq j\}\quad\textrm{for $k=1,2,\ldots,n-1$},$$
where $n_x(Y_k)$ is the number of the symbol $X$ in $Y_k=(y_1,y_2,\ldots,y_k)$.
\end{definition}

For example, $(X,X,A_{1,2})$ is a $\CLS$-sequence, while $(X,X,A_{1,2},B_3)$ is not since $y_4=B_3$ and $3>n_x(Y_3)=2$.
Let $\CLSS_n$ denote the set of $\CLS$-sequences of length $n$.

The following lemma is a fundamental result.
\begin{lemma}\label{Lemma01}
For $n\geq 1$, we have
$\LS(n,k)=\#\{Y_n\in\CLSS_n\mid n_x(Y_n)=k\}$.
\end{lemma}
\begin{proof}
Let $$\CLSS(n,k)=\{Y_n\in\CLSS_n\mid n_x(Y_n)=k\}.$$
Now we start to construct a bijection, denoted by $\Phi$, between $\LSS(n,k)$ and $\CLSS(n,k)$.
When $n=1$, we have $y_1=X$. Set $\Phi(Y_1)=\{1,\overline{1}\}<>$. This gives a bijection from $\CLSS(1,1)$ to $\LSS(1,1)$.
Let $n=m$. Suppose $\Phi$ is a bijection from $\CLSS(n,k)$ to $\CLSS(n,k)$ for all $k$.
Consider the case $n=m+1$. Let $$Y_{m+1}=(y_1,y_2,\ldots,y_m,y_{m+1})\in \CLSS_{m+1}.$$
Then $Y_m=(y_1,y_2,\ldots,y_m)\in \CLSS(m,k)$ for some $k$.
Assume $\Phi(Y_m)\in\LSS(m,k)$. Consider the following three cases:
\begin{enumerate}
  \item [\rm ($i$)] If $y_{m+1}=X$, then let $\Phi(Y_{m+1})$ be obtained from $\Phi(Y_m)$ by putting the box $\{m+1,\overline{m+1}\}$ just before the zero box. In this case, $\Phi(Y_{m+1})\in\LSS(m+1,k+1)$.
  \item [\rm ($ii$)] If $y_{m+1}=A_{i,j}$, then let $\Phi(Y_{m+1})$ be obtained from $\Phi(Y_m)$ by inserting the entry $m+1$ to the $i$th nonzero box and inserting the entry $\overline{m+1}$ to the $j$th nonzero box. In this case, $\Phi(Y_{m+1})\in\LSS(m+1,k)$.
  \item [\rm ($iii$)] If $y_{m+1}=B_{s}$ (resp. $y_{m+1}=\overline{B}_{s}$), then let $\Phi(Y_{m+1})$ be obtained from $\Phi(Y_m)$ by inserting the entry $m+1$ (resp. $\overline{m+1}$) to the $s$th nonzero box and inserting the entry $\overline{m+1}$ (resp. $m+1$) to the zero box. In this case, $\Phi(Y_{m+1})\in\LSS(m+1,k)$.
\end{enumerate}
After the above step, it is clear that the obtained $\Phi(Y_{m+1})$ is in standard form.
By induction, we see that $\Phi$ is the desired bijection from $\CLSS(n,k)$ to $\CLSS(n,k)$,
which also gives a constructive proof of Lemma~\ref{Lemma01}.
\end{proof}

\begin{example}
Let $Y_5=(X,X,A_{2,1},B_2,\overline{B}_1)$.
The correspondence between $Y_5$ and $\Phi(Y_5)$ is built up as follows:
\begin{align*}
X&\Leftrightarrow \{1,\overline{1}\}<>;\\
X&\Leftrightarrow  \{1,\overline{1}\}\{2,\overline{2}\}<>;\\
A_{2,1}&\Leftrightarrow  \{1,\overline{1},\overline{3}\}\{2,\overline{2},3\}<>;\\
B_2&\Leftrightarrow  \{1,\overline{1},\overline{3}\}\{2,\overline{2},3,4\}<\overline{4}>;\\
\overline{B}_1&\Leftrightarrow  \{1,\overline{1},\overline{3},\overline{5}\}\{2,\overline{2},3,4\}<\overline{4},5>.
\end{align*}
\end{example}

As an application of the $\CLS$-sequences, we present the following result.
\begin{lemma}\label{lemma02}
Let $k$ be a given positive integer. Then for $n\geq 1$,
we have
\begin{equation}\label{LSNK01}
\LS(n+k,n)=2^k\sum_{t_k=1}^n\binom{t_k+1}{n}\sum_{t_{k-1}=1}^{t_k}\binom{t_{k-1}+1}{2}\cdots \sum_{t_2=1}^{t_3}\binom{t_2+1}{2}\sum_{t_1=1}^{t_2}\binom{t_1+1}{2}.
\end{equation}
\end{lemma}
\begin{proof}
It follows from Lemma~\ref{Lemma01} that
$$\LS(n+k,n)=\#\{Y_{n+k}\in\CLSS_{n+k}\mid n_x(Y_{n+k})=n\}.$$
Let $Y_{n+k}=y_1y_2\cdots y_{n+k}$ be a given element in $\CLSS_{n+k}$. Since $n_x(Y_{n+k})=n$,
it is natural to assume that $y_i=X$ except $i=t_1+1,t_2+2,\cdots, t_k+k$.
Let $\sigma$ be the corresponding Legendre-Stirling partition of $Y_{n+k}$.
For $1\leq \ell\leq k$, consider the value of $y_{t_\ell+\ell}$. Note that the number of the symbol $X$ before $y_{t_\ell+\ell}$ is $t_\ell$.
Let $\widehat{\sigma}$ be the corresponding Legendre-Stirling partition of $y_1y_2\cdots y_{t_\ell+\ell-1}$.
Now we insert $y_{t_\ell+\ell}$.
We distinguish two cases:
\begin{enumerate}
  \item [\rm ($i$)] If $y_{t_\ell+\ell}=A_{i,j}$, then we should insert the entry $t_\ell+\ell$ to the $i$th nonzero box of $\widehat{\sigma}$ and insert $\overline{t_\ell+\ell}$ to the
  $j$th nonzero box. This gives $2\binom{t_\ell}{2}$ possibilities, since $1\leq i,j\leq t_\ell$ and $i\neq j$.
  \item [\rm ($ii$)] If $y_{t_\ell+\ell}=B_{s}$ (resp. $y_{t_\ell+\ell}=\overline{B}_{s}$), then we should insert the entry $t_\ell+\ell$ (resp. $\overline{t_\ell+\ell}$) to the $s$th nonzero box of $\widehat{\sigma}$ and insert $\overline{t_\ell+\ell}$ (resp. ${t_\ell+\ell}$) to the
  zero box. This gives $2\binom{t_\ell}{1}$ possibilities, since $1\leq s\leq t_\ell$.
\end{enumerate}
Therefore, there are exactly $2\binom{t_\ell}{2}+2\binom{t_\ell}{1}=2\binom{t_\ell+1}{2}$ Legendre-Stirling partitions of $\MM_{t_\ell+\ell}$ can be generated from $\widehat{\sigma}$ by inserting the entry $y_{t_\ell+\ell}$. Note that $1\leq t_{j-1}\leq t_j\leq n$ for $2\leq j\leq k$. Applying the product rule for counting, we immediately get~\eqref{LSNK01}.
\end{proof}

The following simple result will be used
in our discussion.
\begin{lemma}\label{lemma03}
Let $a$ and $b$ be given integers. Then
\begin{equation*}
\binom{x-b}{2}\binom{x}{a}=\binom{a+2}{2}\binom{x}{a+2}+(a+1)(a-b)\binom{x}{a+1}+\binom{a-b}{2}\binom{x}{a}.
\end{equation*}
In particular,
\begin{equation*}
\binom{x-1}{2}\binom{x}{a}=\binom{a+2}{2}\binom{x}{a+2}+(a^2-1)\binom{x}{a+1}+\binom{a-1}{2}\binom{x}{a}.
\end{equation*}
\end{lemma}
\begin{proof}
Note that
$$\binom{a+2}{2}\frac{(x-a)(x-a-1)}{(a+2)(a+1)}+(a+1)(a-b)\frac{x-a}{a+1}+\binom{a-b}{2}=\binom{x-b}{2}.$$
This yields the desired result.
\end{proof}

We can now conclude the main result of this paper from the discussion above.
\begin{theorem}\label{thm01}
Let $k$ be a given nonnegative integer. Then for $n\geq 1$,
the numbers $\LS(n+k,n)$ can be expanded in the binomial basis as
\begin{equation}\label{LSNK-binomial}
\LS(n+k,n)=2^k\sum_{i=k+2}^{3k}\gamma(k,i)\binom{n+k+1}{i},
\end{equation}
where the coefficients $\gamma(k,i)$ are all positive integers for $k+2\leq i\leq 3k$ and satisfy the recurrence relation
\begin{equation}\label{LSNK-recu}
\gamma(k+1,i)=\binom{i-k-1}{2}\gamma(k,i-1)+(i-1)(i-k-2)\gamma(k,i-2)+\binom{i-1}{2}\gamma(k,i-3),
\end{equation}
with the initial conditions $\gamma(0,0)=1$, $\gamma(0,i)=\gamma(i,0)=0$ for $i\neq 0$.
Let $\gamma_k(x)=\sum_{i=k+2}^{3k}\gamma(k,i)x^i$. Then the polynomials $\gamma_k(x)$ satisfy the recurrence relation
\begin{equation}\label{LSNK-recu02}
\gamma_{k+1}(x)=\left(\frac{k(k+1)}{2}-kx+x^2\right)x\gamma_k(x)-
(k+(k-2)x-2x^2)x^2\gamma_k'(x)+\frac{(1+x)^2x^3}{2}\gamma_k''(x),
\end{equation}
with the initial conditions $\gamma_0(x)=1,\gamma_1(x)=x^3$ and $\gamma_2(x)=x^4+8x^5+10x^6$.
\end{theorem}
\begin{proof}
We prove~\eqref{LSNK-binomial} by induction on $k$.
It is clear that $$\LS(n,n)=1=\binom{n+1}{0}.$$
When $k=1$, by using the {\it Chu Shih-Chieh's identity}
$$\binom{n+1}{k+1}=\sum_{i=k}^n\binom{i}{k},$$
we obtain
$$\sum_{t_1=1}^{n}\binom{t_1+1}{2}=\binom{n+2}{3},$$
and so~\eqref{LSnk01} is established.
When $k=2$, it follows from Lemma~\ref{lemma02} that
\begin{align*}
\LS(n+2,n)&=4\sum_{t_2=1}^{n}\binom{t_2+1}{2}\sum_{t_1=1}^{t_2}\binom{t_1+1}{2}\\
&=4\sum_{t_2=1}^{n}\binom{t_2+1}{2}\binom{t_2+2}{3}.
\end{align*}
Setting $x=t_2+2$ and $a=3$ in Lemma~\ref{lemma03}, we get
\begin{align*}
\LS(n+2,n)&=4\sum_{t_2=1}^{n}\left(10\binom{t_2+2}{5}+8\binom{t_2+2}{4}+\binom{t_2+2}{3}\right)\\
&=4\left(10\binom{n+3}{6}+8\binom{n+3}{5}+\binom{n+3}{4}\right),
\end{align*}
which yields~\eqref{LSnk02}. Along the same lines, it is not hard to verify that
\begin{align*}
\LS(n+3,n)&=8\sum_{t_3=1}^{n}\binom{t_3+1}{2}\left(10\binom{t_3+3}{6}+8\binom{t_3+3}{5}+\binom{t_3+3}{4}\right)\\
&=8\left(280\binom{n+4}{9}+448\binom{n+4}{8}+219\binom{n+4}{7}+34\binom{n+4}{6}+\binom{n+4}{5}\right).
\end{align*}
Hence the formula~\eqref{LSNK-binomial} holds for $k=0,1,2,3$, so we proceed to the inductive step.
For $k\geq 3$, assume that
\begin{equation*}\label{LSNK02}
\LS(n+k,n)=2^k\sum_{i=k+2}^{3k}\gamma(k,i)\binom{n+k+1}{i}.
\end{equation*}
It follows from Lemma~\ref{lemma02} that
\begin{align*}
\LS(n+k+1,n)&=2^{k+1}\sum_{t_{k+1}=1}^{n}\binom{t_{k+1}+1}{2}\sum_{i=k+2}^{3k}\gamma(k,i)\binom{t_{k+1}+k+1}{i}
\end{align*}
By using Lemma~\ref{lemma03}, it is routine to verify that the coefficients $\gamma(k,i)$ satisfy the recurrence relation~\eqref{LSNK-recu}, and so~\eqref{LSNK-binomial} is established for general $k$.
Multiplying both sides of~\eqref{LSNK-recu} by $x^i$ and summing for all $i$,
we immediately get~\eqref{LSNK-recu02}.
\end{proof}

In~\cite{Andrews11}, Andrews et al. introduced the {\it (unsigned) Legendre-Stirling numbers $\Lc(n,k)$ of the first kind}, which may be defined by the recurrence relation
$$\Lc(n,k)=\Lc(n-1,k-1)+n(n-1)\Lc(n-1,k),$$
with the initial conditions $\Lc(n,0)=\delta_{n,0}$ and $\Lc(0,n)=\delta_{0,n}$.
Let $f_k(n)=\LS(n+k,n)$.
According to Egge~\cite[Eq.~(23)]{Egge10}, we have
\begin{equation}\label{First-second-Legendre}
\Lc(n-1,n-k-1)=(-1)^kf_k(-n)
\end{equation}
for $k\geq 0$. For $m,k\in \N$, we define $$\binom{-m}{k}=\frac{(-m)(-m-1)\cdots (-m-k+1)}{k!}.$$
Combining~\eqref{LSNK-binomial} and~\eqref{First-second-Legendre}, we immediately get the following result.
\begin{corollary}
Let $k$ be a given nonnegative integer.
For $n\geq 1$,
the numbers $\Lc(n-1,n-k-1)$ can be expanded in the binomial basis as
\begin{equation}
\Lc(n-1,n-k-1)=(-1)^k2^k\sum_{i=k+2}^{3k}\gamma(k,i)\binom{-n+k+1}{i},
\end{equation}
where the coefficients $\gamma(k,i)$ are defined by~\eqref{LSNK-recu}.
\end{corollary}

It follows from~\eqref{LSNK-recu02} that
\begin{align*}
\gamma(k+1,k+3)&=\left(\frac{k(k+1)}{2}-k(k+2)+\frac{(k+2)(k+1)}{2}\right)\gamma(k,k+2),\\
\gamma(k+1,3k+3)&=\left(1+6k+\frac{3k(3k-1)}{2}\right)\gamma(k,3k),\\
\gamma_{k+1}(-1)&=-\left(\frac{k(k+1)}{2}+k+1\right)\gamma_k(-1).
\end{align*}
Since $\gamma(1,3)=1$ and $\gamma_1(-1)=-1$, it is easy to verify that for $k\geq 1$, we have
$$\gamma(k,k+2)=1,~\gamma(k,3k)=\frac{(3k)!}{k!(3!)^k},~\gamma_{k}(-1)=(-1)^k\frac{(k+1)!k!}{2^k}.$$
It should be noted that the number $\gamma(k,3k)$ is
the number of partitions of $\{1,2,\ldots,3k\}$ into blocks of size 3 (see~\cite[A025035]{Sloane}),
and the number $\frac{(k+1)!k!}{2^k}$ is the product of first $k$ positive triangular numbers (see~\cite[A006472]{Sloane}).
Moreover, if the number $\LS(n+k,n)$ is viewed as a polynomial in $n$, then its degree is $3k$, which is implied by the quantity
$\binom{n+k+1}{3k}$. Furthermore, the leading coefficient of $\LS(n+k,n)$ is given by
$$2^k\gamma(k,3k)\frac{1}{(3k)!}=2^k\frac{(3k)!}{k!(3!)^k}\frac{1}{(3k)!}=\frac{1}{k!3^k},$$
which yields~\cite[Theorem 3.1]{Egge10}.
\section{Grammatical interpretations of Jacobi-Stirling numbers of both kinds}\label{Section03}
In this section, a context-free grammar is in the sense of Chen~\cite{Chen93}:
for an alphabet $A$, let $\mathbb{Q}[[A]]$ be the rational commutative ring of formal power
series in monomials formed from letters in $A$. A context-free grammar over
A is a function $G: A\rightarrow \mathbb{Q}[[A]]$ that replace a letter in $A$ by a formal function over $A$.
The formal derivative $D$ is a linear operator defined with respect to a context-free grammar $G$. More precisely,
the derivative $D=D_G$: $\mathbb{Q}[[A]]\rightarrow \mathbb{Q}[[A]]$ is defined as follows:
for $x\in A$, we have $D(x)=G(x)$; for a monomial $u$ in $\mathbb{Q}[[A]]$, $D(u)$ is defined so that $D$ is a derivation,
and for a general element $q\in\mathbb{Q}[[A]]$, $D(q)$ is defined by linearity.
The reader is referred to~\cite{Chen17,Ma1801} for recent progress on this subject.

Let $[n]=\{1,2,\ldots,n\}$. The Stirling number $\Stirling{n}{k}$ of the second kind is the
number of ways to partition $[n]$ into $k$ blocks.
Chen~\cite[Eq. 4.8]{Chen93} showed that if $G=\{x\rightarrow xy, y\rightarrow y\}$,
then
\begin{equation*}
D^n(x)=x\sum_{k=0}^n\Stirling{n}{k}y^k.
\end{equation*}
Let $\msn$ be the symmetric group of all permutations of $[n]$. Let $\cyc(\pi)$ be the number of cycles of $\pi$.
The (unsigned) {\it Stirling number of the first kind} is defined by $$\stirling{n}{k}=\#\{\pi\in\msn\mid \cyc(\pi)=k\}.$$
From~\cite[Eq.~4.8]{Ma131}, we see that if $G=\{x\rightarrow xy, y\rightarrow yz,z\rightarrow z^2\}$,
then
\begin{equation*}
D^n(x)=x\sum_{k=0}^n\stirling{n}{k}y^kz^{n-k}.
\end{equation*}

According to~\cite[Theorem 4.1]{Everitt07}, the {\it Jacobi-Stirling number $\JS_n^k(z)$ of the second kind} is defined by
\begin{equation}\label{xnJsnkz}
x^n=\sum_{k=0}^n\JS_n^k(z)\prod_{i=0}^{k-1}(x-i(z+i)).
\end{equation}
It follows from~\eqref{xnJsnkz} that
the numbers $\JS_n^k(z)$ satisfy the recurrence relation
\begin{align*}
\JS_n^k(z)=\JS_{n-1}^{k-1}(z)+k(k+z)\JS_{n-1}^k(z),
\end{align*}
with the initial conditions $\JS_n^0(z)=\delta_{n,0}$ and $\JS_0^k(z)=\delta_{0,k}$. It is clear that $\JS_n^k(1)=\LS(n,k)$.
Following~\cite[Eq.~(1.3), Eq.~(1.5)]{Zeng10}, {\it the (unsigned) Jacobi-Stirling number $\Jc_n^k(z)$ of the first kind} is defined by
\begin{equation*}\label{Jcnk-def}
\prod_{i=0}^{n-1}\left(x+i(z+i)\right)=\sum_{k=0}^n\Jc_n^k(z)x^k,
\end{equation*}
and the numbers $\Jc_n^k(z)$ satisfy the following recurrence relation
\begin{align*}
\Jc_n^k(z)=\Jc_{n-1}^{k-1}(z)+(n-1)(n-1+z)\Jc_{n-1}^k(z),
\end{align*}
with the initial conditions $\Jc_n^0(z)=\delta_{n,0}$ and $\Jc_0^k(z)=\delta_{k,0}$.
In particular, $\Jc_n^k(1)=\Lc(n,k)$.

Properties and combinatorial interpretations of the Jacobi-Stirling numbers of both kinds were extensively studied in~\cite{Andrews13,Zeng10,Gessel12,Lin14,Merca15,Mongelli12,Mongelli1202}.
The Jacobi-Stirling numbers share many similar properties to those of the
Stirling numbers. A question arises immediately: are there grammatical
descriptions of the Jacobi-Stirling numbers of both kinds? In this section, we give the answer.

As a variant of the $\CLS$-sequence,
we now introduce a marked scheme for Legendre-Stirling partitions.
Given a Legendre-Stirling partition $\sigma=B_1B_2\cdots B_kB_0\in\LSS(n,k)$, where $B_0$ is the zero box of $\sigma$. We mark the box vector $(B_1,B_2,\ldots,B_k)$ by the label $a_k$.
We mark any box pair $(B_i,B_j)$ by a label $b$ and mark any box pair $(B_s,B_0)$ by a label $c$, where $1\leq i<j\leq k$ and $1\leq s\leq k$.
Let $\sigma'$ denote the Legendre-Stirling partition that generated from $\sigma$ by inserting $n+1$ and $\overline{n+1}$.
If $n+1$ and $\overline{n+1}$ are in the same box, then
$$\sigma'=B_1B_2\cdots B_kB_{k+1}B_0,$$
where $B_{k+1}=\{n+1,\overline{n+1}\}$. This case corresponds to the operator $a_k\rightarrow a_{k+1}b^kc$.
If $n+1$ and $\overline{n+1}$ are in different boxes, then
we distinguish two cases:
\begin{enumerate}
  \item [\rm ($i$)] Given a box pair $(B_i,B_j)$, where $1\leq i<j\leq k$. We can put $n+1$ (resp. $\overline{n+1}$) into the box $B_i$ and put $\overline{n+1}$ (resp. $n+1$) into the box $B_j$. This case corresponds to the operator $b\rightarrow 2b$.
   \item [\rm ($ii$)] Given a box pair $(B_i,B_0)$, where $1\leq i\leq k$. We can put $n+1$ (resp. $\overline{n+1}$) into the box $B_i$ and put $\overline{n+1}$ (resp. $n+1$) into the zero box $B_0$. Moreover, we mark any barred entry in the zero box $B_0$ by a label $z$. This case corresponds to the operator $c\rightarrow (1+z)c$.
\end{enumerate}

Let $A=\{a_0,a_1,a_2,a_3,\ldots,b,c\}$ be a set of alphabet.
Using the above marked scheme, it is natural to
consider the following grammars:
\begin{equation}\label{Jsnk-grammar}
G_k=\{a_0\rightarrow a_1c,a_1\rightarrow a_2bc,\ldots,a_{k-1}\rightarrow a_{k}b^{k-1}c,b\rightarrow 2b,c\rightarrow (1+z)c\}
\end{equation}
where $k\geq 1$.

\begin{theorem}\label{thm-grammar01}
Let $G_k$ be the grammars defined by~\eqref{Jsnk-grammar}. Then we have
\begin{equation*}
D_nD_{n-1}\cdots D_1(a_0)=\sum_{k=1}^n\JS_n^k(z)a_kb^{\binom{k}{2}}c^k.
\end{equation*}
\end{theorem}
\begin{proof}
Note that $D_1(a_0)=a_1c$ and $D_2D_1(a_0)=a_2bc^2+(1+z)a_1c$. Thus the result holds for
$n=1,2$. For $m\geq 2$, we define $P_m^k(z)$ by
\begin{equation*}
D_mD_{m-1}\cdots D_1(a_0)=\sum_{k=1}^nP_m^k(z)a_kb^{\binom{k}{2}}c^k.
\end{equation*}
We proceed by induction.
Consider the case $n=m+1$.
Since $$D_{m+1}D_mD_{m-1}\cdots D_1(a_0)=D_{m+1}(D_mD_{m-1}\cdots D_1(a_0)),$$ it follows that
\begin{align*}
D_{m+1}D_{m}\cdots D_1(a_0)&=D_{m+1}\left(\sum_{k=1}^nP_m^k(z)a_kb^{\binom{k}{2}}c^k\right)\\
&=\sum_{k=1}^nP_m^k(z)\left(a_{k+1}b^{\binom{k+2}{2}}c^{k+1}+k(k-1)a_kb^{\binom{k}{2}}c^k+(1+z)ka_kb^{\binom{k}{2}}c^k\right).
\end{align*}
Therefore, we obtain
$P_{m+1}^k(z)=P_m^{k-1}(z)+k(k+z)P_m^k(z)$.
Since the numbers $P_n^k(z)$ and $\JS_n^k(z)$ satisfy the same recurrence relation and initial conditions, so they agree.
\end{proof}

Combining the marked scheme for Legendre-Stirling partitions and Theorem~\ref{thm-grammar01}, it is clear that for $n\geq k$, the number
$\JS_n^k(z)$ is a polynomial of degree $n-k$ in $z$, and the coefficient $z^i$ of $\JS_n^k(z)$ is the number of Legendre-Stirling partitions in $\LSS(n,k)$ with exactly
$i$ barred entries in the zero box, which gives a proof of~\cite[Theorem 2]{Zeng10}.

We end this section by giving the following result.
\begin{theorem}\label{thm-Jcnk}
Let $A=\{a,b_0,b_1,\ldots\}$ be a set of alphabet.
Let $G_k$ be the grammars defined by
\begin{equation*}\label{Jcnk-grammar}
G_k=\{a\rightarrow (k-1)(k-1+z)a,b_{0}\rightarrow b_{1},b_1\rightarrow b_2,\ldots,b_{k-1}\rightarrow b_k\},
\end{equation*}
where $k\geq 1$. Then we have
\begin{equation*}
D_n D_{n-1}\cdots D_1(ab_0)=a\sum_{k=1}^n\Jc_n^k(z)b_k.
\end{equation*}
\end{theorem}
\begin{proof}
Note that $D_1(ab_0)=ab_1$ and $D_2D_1(ab_0)=(1+z)ab_1+ab_2$.
Hence the result holds for $n=1,2$. For $m\geq 2$, we define $Q_m^k(z)$ by
$D_m D_{m-1}\cdots D_1(ab_0)=a\sum_{k=1}^mQ_m^k(z)b_k$.
We proceed by induction.
Consider the case $n=m+1$.
Since $$D_{m+1}D_mD_{m-1}\cdots D_1(ab_0)=D_{m+1}(D_mD_{m-1}\cdots D_1(ab_0)),$$ it follows that
\begin{align*}
D_{m+1}D_{m}\cdots D_1(ab_0)&=D_{m+1}\left(a\sum_{k=1}^mQ_m^k(z)b_k\right)\\
&=a\sum_{k=1}^mQ_m^k(z)m(m+z)b_k+a\sum_{k=1}^mQ_m^kb_{k+1}.
\end{align*}
Therefore, we obtain
$Q_{m+1}^k(z)=Q_m^{k-1}(z)+m(m+z)Q_m^k(z)$.
Since the numbers $Q_n^k(z)$ and $\Jc_n^k(z)$ satisfy the same recurrence relation and initial conditions, so they agree.
\end{proof}
\section{Concluding remarks}
Note that the Jacobi-Stirling numbers are polynomial refinements of the Legendre-Stirling numbers.
It would be interesting to explore combinatorial expansions of Jacobi-Stirling numbers of both kinds.

Let $\gamma_k(x)$ be the polynomials defined by~\eqref{LSNK-recu02}.
We end our paper by proposing the following.
\begin{conjecture}
For any $k\geq 1$, the polynomial $\gamma_k(x)$ has only real zeros.
Set $$\gamma_k(x)=\gamma(k,3k)x^{k+2}\prod_{i=1}^{2k-2}(x-r_i),~\gamma_{k+1}(x)=\gamma(k+1,3k+3)x^{k+3}\prod_{i=1}^{2k}(x-s_i),$$
where $r_{2k-2}<r_{2k-3}<\cdots<r_2<r_1$ and $s_{2k}<s_{2k-1}<s_{2k-2}<\cdots<s_2<s_1$. Then
\begin{equation*}\label{zeros}
s_{2k}<r_{2k-2}<s_{2k-1}<r_{2k-3}<s_{2k-2}<\cdots<r_k<s_{k+1}<s_k<r_{k-1}<\cdots<s_2<r_1<s_1,
\end{equation*}
in which the zeros $s_{k+1}$ and $s_k$ of $\gamma_{k+1}(x)$ are continuous appearance, and the other zeros of $\gamma_{k+1}(x)$ separate the zeros of $\gamma_{k}(x)$.
\end{conjecture}


\end{document}